\documentclass[runningheads,a4paper]{llncs}

\usepackage{amssymb}
\usepackage{graphicx}
\usepackage{url}
\newcommand{\keywords}[1]{\par\addvspace\baselineskip
\noindent\keywordname\enspace\ignorespaces#1}

\usepackage{amsfonts,amscd}
\usepackage{color}
\usepackage{hyperref,multirow}

\newcommand{\bC}{\mathbb{C}}
\newcommand{\bR}{\mathbb{R}}
\newcommand{\bP}{\mathbb{P}}
\def\rank{{\mathop{\rm rank~}\nolimits}}
\def\conj{{\mathop{\rm conj}\nolimits}}
\newcommand{\eqref}[1]{(\ref{#1})}
\newcommand{\Cert}{{\hbox{\bf Certify}}}

\newcommand{\vppp}{\widehat{\mathbb{P}}^3}

\renewcommand{\:}{\colon}

\usepackage{fixme}

\begin{document}

\title{Certifying reality of projections}
\author{Jonathan D. Hauenstein\inst{1}, 
Avinash Kulkarni\inst{2}, \\
Emre C. Sert\"oz\inst{3}, \and
Samantha N. Sherman\inst{1}}

\authorrunning{Hauenstein-Kulkarni-Sert\"oz-Sherman}
\institute{\scriptsize
Department of Applied and Computational Mathematics and Statistics, \\
University of Notre Dame, USA\\
\email{\{hauenstein,ssherma1\}@nd.edu, \texttt{www.nd.edu/$\sim$jhauenst}}
\thanks{JDH was supported by Sloan Fellowship BR2014-110 TR14 and NSF ACI-1460032.  
SNS was supported by Schmitt Leadership Fellowship in Science~and~Engineering.}
\and
Department of Mathematics, 
Simon Fraser University, Canada \\
\email{akulkarn@sfu.ca}
\thanks{AK was partially supported by the Max Planck Institute for Mathematics in the Sciences}
\and 
Max Planck Institute for Mathematics in the Sciences, Leipzig, Germany \\
\email{emresertoz@gmail.com, \texttt{www.emresertoz.com}}
}

\maketitle
\begin{abstract}
Computational tools in numerical algebraic geometry can be used to numerically
approximate solutions to a system of polynomial equations.  
If the system is well-constrained (i.e., square), Newton's method
is locally quadratically convergent near each nonsingular solution.  
In such cases, Smale's alpha theory can be used to certify that a given
point is in the quadratic convergence basin of some solution.  
This was extended to certifiably determine the reality of the corresponding 
solution when the polynomial system is real.  Using the theory of Newton-invariant sets,
we certifiably decide the reality of projections of solutions.  
We apply this method to certifiably count the number of 
real and totally real tritangent planes for instances of curves of genus~4.  
\end{abstract}

\section{Introduction}\label{Sec:Intro}

For a well-constrained system of polynomial equations $f$, 
numerical algebraic geometric tools (see, e.g., \cite{BHSW,SW05})
can be used to compute numerical approximations of solutions of $f = 0$.  
These approximations can be certified to lie in a quadratic convergence
basin of Newton's method applied to $f$ using
Smale's $\alpha$-theory (see, e.g., \cite[Chap.~8]{BCSS}).
When the system $f$ is real, $\alpha$-theory can be used to certifiably
determine if the true solution corresponding to an approximate solution is real~\cite{alphaCertified}.
That is, one can certifiably decide whether or not every coordinate of a 
solution is real from a sufficiently accurate approximation. 
It is often desirable in computational algebraic geometry to instead decide the reality of a projection of a solution of a real polynomial system. In this manuscript, we develop an approach for this situation using Newton-invariant sets \cite{NewtonInvariant}.

The paper is organized as follows.
Section~\ref{Sec:alphaTheory} provides a summary of Smale's $\alpha$-theory
and Newton-invariant sets.  Section~\ref{Sec:Certification} provides
our main results regarding certification of reality of projections.
Section~\ref{Sec:Tritangents} applies the method to
certifying real and totally real tritangents of various genus~$4$ curves.

\section{Smale's alpha theory and Newton-invariant sets}\label{Sec:alphaTheory}

Our certification procedure is based on the ability to certify
quadratic convergence of Newton's method via Smale's $\alpha$-theory 
(see, e.g., \cite[Chap.~8]{BCSS}) and Newton-invariant sets \cite{NewtonInvariant}.
This section summarizes these two items following \cite{NewtonInvariant}.

Assume that $f\: \bC^n\rightarrow\bC^n$ is an
analytic map and consider the Newton iteration map
$N_f\: \bC^n\rightarrow\bC^n$ defined by
{\small
$$N_f(x) := \left\{\begin{array}{ll} x - Df(x)^{-1} f(x) & \hbox{~~~if $Df(x)$ is invertible,} \\ x & \hbox{~~~otherwise,} \end{array}\right.$$
}where $Df(x)$ is the Jacobian matrix of $f$ at $x$.
The map $N_f$ is globally defined with fixed points $\{x\in\bC^n~|~f(x) = 0 \hbox{~~or~~}\rank Df(x) < n\}$. Hence, if $Df(x)$ is invertible and $N_f(x) = x$, then $f(x) = 0$.

One aims to find solutions of $f = 0$ by iterating $N_f$ to locate fixed points.
To that end, for each $k\geq1$, define $N_f^k(x) := \underbrace{N_f \circ \cdots \circ N_f}_{\hbox{\scriptsize $k$ times}}(x)$.

\begin{definition}\label{def:approxSoln}
A point $x\in\bC^n$ is an {\em approximate solution} of $f = 0$ if there exists~$\xi\in\bC^n$ such that $f(\xi) = 0$ and $\|N_f^k(x) - \xi\| \leq \left(\frac{1}{2}\right)^{2^k-1} \|x - \xi\|$ for each~$k\geq1$ where $\|\cdot\|$ is the Euclidean norm on $\bC^n$.
The point $\xi$ is the {\em associated solution} to $x$
and the sequence $\{N_f^k(x)\}_{k\geq0}$ converges quadratically~to~$\xi$.
\end{definition}

Smale's $\alpha$-theory provides sufficient conditions 
for $x$ to be an approximate solution of $f = 0$
via data computable from $f$ and $x$.
We will use approximate solutions to determine characteristics 
of the corresponding associated solutions using Newton-invariant sets.

\begin{definition}\label{def:NewtonInvariant}
A set $V\subset\bC^n$ is called {\em Newton invariant} with respect to $f$
if $N_f(v)\in V$ for every $v\in V$ and $\lim_{k\rightarrow\infty} N_f^k(v) \in V$
for every $v\in V$ such that this limit exists.  
\end{definition}

For example, the set $V = \bR^n$ is Newton invariant with respect to a real map~$f$.
The algorithm presented in Section~\ref{Sec:Certification} 
considers both the set of real numbers as well as other Newton-invariant sets
to perform certification together with the following theorem 
derived from~\cite[Ch.~8]{BCSS} and \cite{NewtonInvariant}.

\begin{theorem}\label{Thm:alpha}
Let $f\:\bC^n\rightarrow\bC^n$ be analytic, let $V\subset\bC^n$ be Newton invariant with respect to $f$, let
$x,y\in\bC^n$ such that $Df(x)$ and $Df(y)$ are invertible,~and let 
{\small
$$\begin{array}{rclcrcl}
\alpha(f,x) & := & \beta(f,x)\cdot\gamma(f,x), & &
                     \beta(f,x) & := & \|x - N_f(x)\| = \|Df(x)^{-1} f(x)\|,  \\ \rule{0pt}{16pt}
                     \gamma(f,x) & := & {\scriptsize 
\displaystyle\sup_{k\geq 2} \left\|\frac{Df(x)^{-1} D^kf(x)}{k!}\right\|^{\frac{1}{k-1}}}, & & 
                     \delta_V(x) & := & \displaystyle\inf_{v\in V}\|x - v\|
                     \end{array}$$
}where the norms are the corresponding vector and operator Euclidean norms.
\begin{enumerate}
\item\label{Item:1} If $4\cdot\alpha(f,x) < 13 - 3\sqrt{17}$, then $x$ is an approximate solution of $f = 0$.
\item\label{Item:2} If $100\cdot \alpha(f,x) < 3$ and $20\cdot \|x - y\|\cdot\gamma(f,x) < 1$, then $x$ and $y$ are approximate solutions of $f = 0$ with the same associated solution.
\item\label{Item:3} Suppose that $x$ is an approximate solution of $f = 0$ with associated solution~$\xi$.
\begin{enumerate}
\item\label{Item:3a} $N_f(x)$ is also an approximate solution with associated solution $\xi$ and 
$$\|x - \xi\| \leq 2 \beta(f,x) = 2\|x - N_f(x)\| = 2\|Df(x)^{-1}f(x)\|.$$
\item\label{Item:3b} If $\delta_V(x) > 2\beta(f,x)$, then $\xi\notin V$.
\item\label{Item:3c} If $100\cdot \alpha(f,x) < 3$ and $20\cdot \delta_V(x)\cdot\gamma(f,x) < 1$, then $\xi \in V$.
\end{enumerate}
\end{enumerate}
\end{theorem}

The value $\beta(f,x)$ is the Newton residual.  
When $f$ is a polynomial system,~$\gamma(f,x)$ is a maximum over finitely many terms and thus can be easily bounded above~\cite{SS93}. 
A similar bound for polynomial-exponential systems can be found in~\cite{HL}.  The value~$\delta_V(x)$ is the distance
between $x$ and~$V$.  The special case of $V = \bR^n$ was first
considered in~\cite{alphaCertified}.

The following procedure from \cite{NewtonInvariant},
which is based on Theorem~\ref{Thm:alpha}, certifiably decides
if the associated solution of a given approximate solution 
lies in a given Newton-invariant set $V$.

{\small
\begin{description}
  \item[Procedure $b = \Cert(f,x,\delta_V)$]
  \item[Input] A well-constrained analytic system $f\: \bC^n\rightarrow\bC^n$ such that $\gamma(f,\cdot)$ can be computed (or bounded) algorithmically, a point $x\in\bC^n$ which is an approximate solution of $f = 0$ with associated solution $\xi$ such that $Df(\xi)^{-1}$ exists,
and distance function $\delta_V$ for some Newton-invariant set $V$ that 
can be computed algorithmically.
  \item[Output] A boolean $b$ which is {\tt true} if $\xi\in V$ and {\tt false} if $\xi\notin V$.
  \item[Begin] \hskip -0.1in
  \begin{enumerate}
    \item\label{Step:1} Compute $\beta := \beta(f,x)$, $\gamma := \gamma(f,x)$, $\alpha := \beta\cdot\gamma$, and $\delta := \delta_V(x)$.
    \item\label{Step:2} If $\delta > 2\beta$, {\bf Return} {\tt false}.
    \item\label{Step:3} If $100\cdot \alpha < 3$ and $20\cdot \delta\cdot\gamma < 1$, {\bf Return} {\tt true}.
    \item\label{Step:4} Update $x := N_f(x)$ and go to Step~\ref{Step:1}.
  \end{enumerate}
\end{description}
}

\section{Certification of reality}\label{Sec:Certification}

The systems under consideration are well-constrained polynomial systems
{\footnotesize
\begin{equation}\label{eq:System}
f(a,b_1,\dots,b_k,c_1,\dots,c_\ell,d_1,\dots,d_\ell) = \left[\begin{array}{cl} g(a) \\ p(a,b_i) & \hbox{for~}i=1,\dots,k 
\\ p(a,c_i) & \hbox{for~} i = 1,\dots,\ell \\ p(a,d_i) & \hbox{for~} i = 1,\dots,\ell\end{array}\right]
\end{equation}
}with variables $a\in\bC^m$ and $b_r,c_s,d_t\in\bC^q$,
and polynomial systems $g\: \bC^m\rightarrow\bC^u$ and $p\: \bC^{m+q}\rightarrow\bC^w$ which have 
real coefficients such that 
\begin{equation}\label{eq:Bounds}
u \leq m \hbox{~~~and~~~} m + (k+2\ell)q = u + (k+2\ell)w.
\end{equation}
The first condition in \eqref{eq:Bounds} yields that $a$ is not over-constrained by $g$
while the second condition provides that the whole system is well-constrained.

\begin{example}\label{ex:Illustrative}
To illustrate the setup, we consider an example with
$m = 3$, $k = 0$, $\ell = 1$, $q = 1$, $u = 1$, and $w = 2$
so that \eqref{eq:Bounds} holds, resulting in a well-constrained system of $5$ polynomials
in $5$ variables. Namely, we consider
$$\hbox{\small $
f(a,c,d) = \left[\begin{array}{c} 
g(a) \rule{0pt}{2.6ex} \rule[-1.5ex]{0pt}{0pt} \\ \hline 
p(a,c) \rule{0pt}{2.6ex} \rule[-1.5ex]{0pt}{0pt} \\ \hline \rule{0pt}{2.6ex}
p(a,d)
\end{array}\right]
=$} \hbox{\tiny $
\left[\begin{array}{c}
a_1^2 + a_2^2 + a_3^2 - 1  \rule[-1.8ex]{0pt}{0pt} \\ \hline 
a_1 + (1-c^2)(a_2 c + a_3 c^2) \rule{0pt}{3ex} \\
a_1 (3 c^2 - 1) + a_2 (2 c^5 - 4c^3 + 2c - 1) \rule{0pt}{3ex} \rule[-1.8ex]{0pt}{0pt} \\ \hline 
a_1 + (1-d^2)(a_2 d + a_3 d^2) \rule{0pt}{3ex} \\
a_1 (3 d^2 - 1) + a_2 (2 d^5 - 4d^3 + 2d - 1) \rule{0pt}{3ex} \rule[-1.8ex]{0pt}{0pt} \\
\end{array}\right]$}.$$
\end{example}

Since the polynomial system $f$ in \eqref{eq:System} has real coefficients,
we can use Theorem~\ref{Thm:alpha} with $V = \bR^n$ 
where $n = m+(k+2\ell)q = u+(k+2\ell)w$ to 
certifiably determine if all coordinates of the associated solution are simultaneously real.

\begin{example}\label{ex:Illustrative2}
Let $f$ be the polynomial system with real coefficients considered in Ex.~\ref{ex:Illustrative}
with Newton-invariant set $V = \bR^5$.
For the points $P_1$ and $P_2$, respectively:
$$
\begin{array}{l}
\hbox{\tiny $\displaystyle\left(
\frac{1543}{8003} + \frac{\sqrt{-1}}{530485174},
\frac{-34488}{50521} - \frac{\sqrt{-1}}{190996265},
\frac{32768}{46489} - \frac{\sqrt{-1}}{310964547},
\frac{6713}{18120} + \frac{4777\sqrt{-1}}{19088},
\frac{6713}{18120} - \frac{4538\sqrt{-1}}{18133}\right)$,}\\[0.1in]
\hbox{\tiny $\displaystyle\left(
\frac{18245}{111912} - \frac{\sqrt{-1}}{772703930},
\frac{15244}{38793} - \frac{\sqrt{-1}}{307556791},
\frac{27099}{29944} - \frac{\sqrt{-1}}{155308656},
\frac{-44817}{40271} - \frac{\sqrt{-1}}{372454657},
\frac{8603}{8149} + \frac{\sqrt{-1}}{608134511}\right),$}
\end{array}
$$
alphaCertified \cite{alphaCertified} computed the following information:
{\footnotesize
$$\begin{array}{c|c|c|c|c}
j~ & ~\hbox{upper bound of~} \alpha(f,P_j)~&~\beta(f,P_j)~&~\hbox{upper bound of~} \gamma(f,P_j)~&~\delta_{\bR^5}(P_j) \\ \hline
1 & 1.32\cdot10^{-5} & 2.05\cdot10^{-8} & 6.40\cdot10^2 & 0.35 \rule{0pt}{1.6ex} \\
2 & 2.38\cdot10^{-4} & 1.47\cdot10^{-8} & 1.63\cdot10^2 & 7.98\cdot10^{-9}
\end{array}$$
}Item~\ref{Item:1} of Theorem~\ref{Thm:alpha} yields 
that both points $P_1$ and $P_2$ 
are approximate solutions of \mbox{$f = 0$}.
Suppose that $\xi_1$ and $\xi_2$, respectively, are the corresponding associated solutions.
Items~\ref{Item:3b} and~\ref{Item:3c}, respectively, provide that $\xi_1\notin\bR^5$ and $\xi_2\in\bR^5$.  
\end{example}

Rather than consider all coordinates simultaneously, the following shows that we can 
certifiably decide the reality of some of the coordinates.

\begin{theorem}\label{Thm:NewtonInv}
For $f$ as in \eqref{eq:System}, the set
\begin{equation}\label{eq:V}
V = \hbox{\small $
\left\{(a,b_1,\dots,b_k,c_1,\dots,c_\ell,\conj(c_1),\dots,\conj(c_\ell)) \in \bR^m \times (\bR^{q})^k \times (\bC^{q})^{2\ell}\right\}$}
\end{equation}
is Newton invariant with respect to $f$ where $\conj()$ denotes complex conjugate.
\end{theorem}
\begin{proof}
Suppose that $v=(a,b_1,\dots,b_k,c_1,\dots,c_\ell,d_1,\dots,d_\ell) \in V$
such that the Jacobian matrix $Df(v)$ is invertible.  Let $\Delta v = Df(v)^{-1}f(v)$ and write
$$\Delta v = \left[\begin{array}{cccccccccc} 
\Delta a^T & \Delta b_1^T & \cdots & \Delta b_k^T & \Delta c_1^T & \cdots & \Delta c_\ell^T &
\Delta d_1^T & \cdots & \Delta d_\ell^T\end{array}\right]^T.$$
Since $f$ has real coefficients, we know that
$$\conj(\Delta v) = \conj(Df(v)^{-1}f(v)) = Df(\conj(v))^{-1} f(\conj(v)).$$
Since $v\in V$, $\conj(v)=(a,b_1,\dots,b_k,d_1,\dots,d_\ell,c_1,\dots,c_\ell) \in V$.
Based on the structure of $f$, it immediately follows that
$$\hbox{\scriptsize $\conj(\Delta v) = 
Df(\conj(v))^{-1} f(\conj(v)) = 
\left[\begin{array}{cccccccccc} 
\Delta a^T & \Delta b_1^T & \cdots & \Delta b_k^T & \Delta d_1^T & \cdots & \Delta d_\ell^T & 
\Delta c_1^T & \cdots & \Delta c_\ell^T \end{array}\right]^T$}.$$
Hence, $\conj(\Delta a) = \Delta a$, $\conj(\Delta b_i) = \Delta b_i$,
and $\conj(\Delta c_j) = \Delta d_j$.  Thus, it immediately follows that
$N_f(v) = v - \Delta v \in V$.  

The remaining condition in Defn.~\ref{def:NewtonInvariant} follows
from the fact that $V$ is closed.\qed
\end{proof}

All that remains to utilize $\Cert$ is to provide 
a formula for $\delta_V$.

\begin{proposition}\label{prop:Distance}
For any $x = (a,b_1,\dots,b_k,c_1,\dots,c_\ell,d_1,\dots,d_\ell)\in\bC^{m+(k+2\ell)q}$ and~$V$ as in \eqref{eq:V},
{\small
\begin{equation}\label{eq:DeltaV}
\hbox{\footnotesize $\delta_V(x) = 
\displaystyle\frac{1}{2}$}
\left\|
\hbox{\footnotesize $
\begin{array}{l}
(a-\conj(a), b_1 - \conj(b_1), \dots, b_k - \conj(b_k), \\
~~c_1 - \conj(d_1), \dots, c_\ell - \conj(d_\ell), d_1 - \conj(c_1), \dots, d_\ell - \conj(c_\ell))
\end{array}$}\right\|.
\end{equation}}
\end{proposition}
\begin{proof}
The projection of $x = (a,b_1,\dots,b_k,c_1,\dots,c_\ell,d_1,\dots,d_\ell)$ onto~$V$~is
{\small 
$$\begin{array}{l}
v = \frac{1}{2}(a+\conj(a), b_1 + \conj(b_1), \dots, b_k + \conj(b_k), \\
~~~~~~~~~~~~~~~c_1 + \conj(d_1), \dots, c_\ell + \conj(d_\ell), d_1 + \conj(c_1), \dots, d_\ell + \conj(c_\ell)).
\end{array}$$
}Thus, $\delta_V(x) = \|x - v\|$ which simplifies to \eqref{eq:DeltaV}.\qed
\end{proof}

\begin{example}\label{ex:Illustrative3}
For the polynomial system $f$ considered in Ex.~\ref{ex:Illustrative},
Theorem~\ref{Thm:NewtonInv} provides that
$V = \{(a,c_1,\conj(c_1))\in\bR^3\times\bC\times\bC\}$
is Newton invariant with respect to $f$.  
Let $\xi_1$ be the associated solution
of the first point $P_1$ from Ex.~\ref{ex:Illustrative}.
Since~$\delta_V(P_1) = 8.88\cdot10^{-9}$,
we know $\xi_1\in V$
using the data from Ex.~\ref{ex:Illustrative2}
together with Item~\ref{Item:3c} of Theorem~\ref{Thm:alpha},
i.e., the first three coordinates of $\xi_1$ are real
and the last two coordinates are complex conjugates of each other.
Hence, $\xi_1\in V\setminus \bR^5$.
\end{example}

\section{Tritangents}\label{Sec:Tritangents}

We conclude by applying this new certification method 
to a problem from real algebraic geometry considered in \cite{KRNS17,K18}. 
A \emph{smooth space sextic} is a nonsingular algebraic curve $C\subset\bP^3$ which is the intersection of a quadric
surface~$Q$ and cubic surface $\Gamma$. The curve $C$ is a curve of degree $6$ and genus $4$, and every hyperplane of $\bP^3$ intersects $C$ in exactly $6$ points (counting multiplicities). The problem considered in \cite{KRNS17,K18} concerns counting the number of hyperplanes 
which are tangent to $C$ at all points of intersection.

\begin{definition}
A plane $H\subset\bP^3$ is a {\em tritangent plane} for $C$ 
if every point in~$C\cap H$ has even intersection multiplicity.
\end{definition}

In the generic case, each tritangent plane intersects $C$
in $3$ points, each with multiplicity~$2$, and there are a total
of $120$ complex tritangent planes. For simplicity, we henceforth restrict our attention to the generic case. Each of the $120$ tritangent planes can be categorized as either totally real, real, or nonreal.  

\begin{definition}
A tritangent plane $H$ is \emph{real} if it can be expressed as the 
solution set of a linear equation with real coefficients and 
\emph{nonreal} otherwise. A real tritangent plane is 
\emph{totally real} if each point in $C\cap H$ is real.
\end{definition}

\begin{example}\label{ex:Sextic16}
The smooth space sextic curve $C\subset\bP^3$ equal to
$$\{[x_0,x_1,x_2,x_3]\in\bP^3~|~x_0^2+x_0x_3 = x_1x_2, 
~x_0x_2(x_0 + x_1 + x_3) = x_3(x_1^2 - x_2^2 + x_3^2)
\}$$
has $16$ real tritangents, $7$ of which are totally real, 
and $104$ nonreal tritangents.  
\end{example}

\subsection{Counting real and totally real tritangents}

Gross and Harris \cite{GH81} prove that the number of real tritangents of a genus 4 curve is either 0, 8, 16, 24, 32, 64 or 120. This number depends only on the topological properties of the real part of the curve,
as summarized in Table~\ref{tab:Results}.

\begin{example}
Since the curve $C$ in Ex.~\ref{ex:Sextic16} has $16$ real tritangents,
it follows from~\cite{GH81} that the real part of $C$ consists of two connected components.
\end{example}

In contrast, \emph{totally} real tritangents reflect the \emph{extrinsic} geometry of the real part of the curve.
Indeed, Kummer \cite{K18} recently obtained bounds on 
the number of totally real tritangents for each real topological type.
We will use our certification procedure to prove results
that help close the gaps between the theoretical bounds and instances
which have actually been realized.  

To that end, we formulate a
well-constrained parameterized polynomial system 
of the form \eqref{eq:System} as follows.
For a generic smooth space sextic $C=Q\cap\Gamma\subset\bP^3$,
let $q$ and $c$ be quadric and cubic polynomials
that define $Q$ and $\Gamma$, respectively.  By
assuming the coordinates are in general position, we solve in affine space
by setting the first coordinate equal to $1$.  In particular, 
we are seeking $a \in\bC^3$, $x_1,x_2,x_3\in\bC^3$, 
and $\lambda_1, \lambda_2, \lambda_3\in\bC^2$ such that
\begin{equation}\label{eq:TritangentSystem}
\hbox{\scriptsize $
f(h,x_1,\lambda_1,x_2,\lambda_2,x_3,\lambda_3) = \left[\begin{array}{c}
p(h,x_1,\lambda_1) \\p(h,x_2,\lambda_2) \\p(h,x_3,\lambda_3)
\end{array}\right] = 0
\hbox{~~with~~}
p(h,x_i,\lambda_i) = 
\left[\begin{array}{c}
H(X_i) \\
q(X_i) \\
c(X_i) \\
\left[\begin{array}{c} \nabla_x H(X_i) \\ \nabla_x q(X_i) \\ \nabla_x c(X_i) 
\end{array}\right] \Lambda_i
\end{array}
\right]$}
\end{equation}
where $H = [1,h]\in\vppp$, $X_i = [1,x_i]\in\bP^3$, $\Lambda_i = [1,\lambda]\in\bP^2$,
and \hbox{$\nabla_x \zeta([1,x])\in\bC^3$} is the gradient of $\zeta$ with respect to $x$.  
In particular, $f$ is a system of $18$ polynomials in $18$ variables
with the first $3$ polynomials in $p$ enforcing that $X_i\in C\cap H$
and the last $3$ polynomials providing that $H$ is tangent to $C$ at $X_i$.  
The values of $k$ and $\ell$ from \eqref{eq:System}
are dependent on the number of real points in $C\cap H$.  
A real tritangent $H$ will either have three or one real points in $C\cap H$
corresponding, respectively to totally real tritangents ($k = 3$ and $\ell = 0$)
and real tritangents that are not totally real ($k = \ell = 1$).

\begin{remark}\label{remark:Parameter}
For generic quadric $q$ and cubic $c$, the condition $f = 0$ in \eqref{eq:TritangentSystem} has $120\cdot 3! = 720$ isolated solutions
where the factor $3! = 6$ corresponds to trivial reorderings.  By selecting one point
in each orbit, \eqref{eq:TritangentSystem} can be used as a parameter homotopy~\cite{ParameterHomotopy},
where the parameters are the coefficients of $q$ and $c$, to compute tritangents for generic
smooth space sextic curves.
\end{remark}

\subsection{Computational results}\label{Sec:Results}

In the following, we utilize Bertini \cite{Bertini} 
to numerically approximate the tritangents via a parameter
homotopy following Remark~\ref{remark:Parameter}.  
After heuristically classifying the tritangents as either totally real, real, or nonreal, 
we use the results from Section~\ref{Sec:Certification}
applied to $f$ in \eqref{eq:TritangentSystem} to certify the results
using alphaCertified \cite{alphaCertified}.  More computational details 
for applying our approach to the examples that follow 
can be found at \url{http://dx.doi.org/10.7274/R0DB7ZW2}.  
The reported timings are based on using 
either one (in serial) or all 64 (in parallel) cores of a 2.4GHz AMD Opteron Processor 
6378 with 128 GB RAM.  

\begin{example}\label{Ex:16}
For $i=1,2$, let $C_i\subset\bP^3$ be defined by $q_i = c_i = 0$ where
\vspace{-0.05in}
$${\tiny
\begin{array}{l}
q_1(x) = q_2(x) = x_0 x_3 - x_1 x_2 \\[0.03in]
c_1(x) = (25x_0^3 - 24x_0^2x_1 - 89x_0^2x_2 - 55x_0^2x_3 - 14x_1^3 - 31x_1^2x_2 + 86x_1x_2x_3 + 74x_2^2x_3 - 45x_2x_3^2 - 62x_3^3)/100 \\[0.03in]
c_2(x) = (89x_0^3 - 41x_0^2x_1 - 87x_0x_1^2 - 26x_0x_2^2 - 25x_1^2x_2 + 42x_1^2x_3 + 56x_1x_2^2 + 87x_2^3 - 67x_2x_3^2 - 42x_3^3)/100.
\end{array}}
$$
We first use a parameter homotopy in Bertini following Remark~\ref{remark:Parameter}
to numerically approximate the solutions of $f = 0$ in \eqref{eq:TritangentSystem}.
Each of these instances took approximately 45 seconds in serial and 
1.5 seconds in parallel to compute all numerical solutions to roughly
$50$ correct digits.  Converting to rational numbers 
and applying alphaCertified to each instance 
shows that all numerical approximations computed
by Bertini are approximate solutions in roughly $33$ minutes using 
rational arithmetic with serial processing.

First, we certify that we have indeed computed~$120$ distinct 
tritangents up to the action of reordering.  This is accomplished by comparing the pairwise distances between
the $h$ coordinates corresponding to the tritangent hyperplane 
with the known error bound $2\beta$ from Item~\ref{Item:3a} of Theorem~\ref{Thm:alpha}.
In both of our examples, $2\beta < 10^{-54}$ while the pairwise distances were larger than $10^{-2}$
showing that $120$ distinct tritangents were computed as expected.

Second, we compare the size of the imaginary parts of the $h$ coordinates
with the error bound $2\beta$ to certifiably determine which are nonreal tritangets.  
For both cases, this proves that there are $104$ nonreal tritangents leaving 
$16$ tritangents requiring further investigation.  

Third, we apply \Cert~with $V = \bR^{18}$ to certifiably determine the number of totally real tritangents.
This proves that $C_1$ and $C_2$ have exactly~$0$ and $16$ totally real tritangents, respectively.

The only remaining item is to show that the $16$ tritangents for $C_1$ are real
which follows from our new results in Section~\ref{Sec:Certification}.
We reorder the intersection points so that the first one has the smallest imaginary part
and apply \Cert~with $V$ as in \eqref{eq:V} where $k = 1$ and $\ell = 1$, i.e., one real
intersection point and a pair of complex conjugate intersection points.  

In summary, these computations prove that both $C_1$ and $C_2$ have 16 real tritangents, 
where none and all of these $16$ are totally real, respectively.  
\end{example}

Example~\ref{Ex:16} provides two new instances of results that had
not been realized in~\cite{KRNS17}.  Combining these two examples together
with results from \cite{KRNS17,K18} shows that any number 
between $0$ and $16$ totally real tritangents
can be realized for a smooth sextic curve which has $16$ real tritangents.  
In Table 4.2 we summarize the theoretical bounds from \cite{K18} for the number of totally real tritangents, together with the values that are realized in \cite{KRNS17} and our computations (including the computations we describe below). 
In particular, the bold numbers
show new results we obtained using our certification approach. Only
5~open cases remain to be realized or shown to be impossible: 64 real tritangents with all 
64 totally real and 120 real tritangents with 80 -- 83 totally real tritangents.  
\begin{table}[!ht]\label{tab:Results}
{\small
\begin{center}
\begin{tabular}{c|c|c|c|c}
\# real & \# connected &  dividing & range of & realized \# totally real \\
tritangents \cite{GH81} & real components & type? & \# totally real \cite{K18} & (\cite{KRNS17} \& our results) \\
\hline
0 & 0 & No & [0,0] & [0,0] \\
8 & 1 & No & [0,8] & [0,8] \\
16 & 2 & No & [0,16] & [{\bf 0},{\bf 16}] \\
24 & 3 & Yes & [0,24] & [0,{\bf 24}] \\
32 & 3 & No & [8,32] & [{\bf 8},32] \\
64 & 4 & No & [32,64] & [{\bf 32},63] \\
120 & 5 & Yes & [80,120] & [84,120]
\end{tabular}
\end{center}
\caption{Summary of results for tritangents of genus $4$ curves
with {\bf bold} numbers showing the new results obtained using our certification approach.}
}\end{table}

In our computations to generate these results, we started with the Cayley cubic
\mbox{$c = -x_0^2x_2+x_0^2x_3+x_1^2x_2+x_1^2x_3+x_2^2x_3-x_3^3$}~and~selected 
quadrics $q$ which intersected various real components of the Cayley cubic surface $\Gamma$ defined by~$c$.  
We then randomly perturbed all of the coefficients 
of $q$ and $c$ to locally explore the surrounding area of the parameter space
of the selected instance.  As in Ex.~\ref{Ex:16}, Bertini was used to compute
numerical approximations of the solutions with
certification provided by alphaCertified.

\section{Acknowledgments}

The authors thank Mario Kummer and Bernd Sturmfels for many wonderful 
discussions regarding tritangents.

\end{document}